\documentclass[12pt]{article} 
\usepackage[utf8]{inputenc}
\usepackage{mathtools}
\usepackage{amstext}
\usepackage{amsfonts,tikz}
\usepackage{amssymb}
\usepackage{amsthm}
\usepackage{amsmath}
\usepackage{enumitem}
\usepackage{cases}
\usepackage{verbatim}
\usepackage{xcolor}
\usepackage{footmisc}

\usepackage{mathrsfs}
\usepackage[colorlinks=true,linkcolor=blue]{hyperref}
\usepackage[margin=1in]{geometry}

\newcommand{\pr}[1][]{\mathbb{P}}

\newtheorem{thm}{Theorem}[section]

\newtheorem{fact}[thm]{Fact}

\newtheorem{lemma}[thm]{Lemma}

\newcommand{\remove}[1] {}

\author{Christopher Brice\thanks{Columbia University, NYC, NY} \and Erin Meger \thanks{School of Computing, Queen's University, Kingston, ON (erin.meger@queensu.ca)}
\and Nhat-Dinh Nguyen\thanks{CUNY, The City College of New York, NYC, NY} \and Allen Rakhamimov\thanks{The Cooper Union for the
Advancement of Science and Art, NYC, NY} \and
Abigail Raz\thanks{Department of Mathematics, The Cooper Union for the
Advancement of Science and Art, NYC, NY
  (abigail.raz@cooper.edu) - Corresponding Author}}

\title{Zero Forcing on Iterated Graph Models}

\begin{document}
\maketitle

\begin{abstract}
Modeling how information travels throughout a network has vast applications across social sciences, cybersecurity, and graph-based neural networks. In this paper, we consider the zero forcing model for information diffusion on iterative deterministic complex network models. In particular, we continue the exploration of the Iterative Local Transitive (ILT) model and the Iterative Local Anti-Transitive (ILAT) model, both introduced by Bonato et. al. in 2009 and 2017, respectively. These models use ideas from Structural Balance Theory to generate edges through a notion of cloning where ``the friend of my friend is my friend'' and anticloning where ``the enemy of my enemy is my friend.'' 

Zero forcing, introduced independently by Burgarth and Giovanetti and a special working group at AIM in 2007 and 2008, begins with some set of forced vertices, the remaining are unforced. If a forced vertex has a single unforced neighbor, that neighbor becomes forced. The minimum number of vertices in a starting set required to guarantee all vertices eventually become forced is the zero forcing number of the graph. The maximum number of vertices in a starting set such that the graph cannot become fully forced is the failed zero forcing number of a graph. In this paper we provide bounds for both the zero forcing and failed zero forcing number of graphs created using both the ILT and ILAT models. In particular, we show that in the case of ILAT graphs the failed zero forcing number can only take on one of four values, dependent on the number of vertices of the graph. Finally, we briefly consider another information diffusion model, graph burning, on a more general iterated graph model. 
\end{abstract}

\section{Introduction} 

Models of complex networks are used to predict
the evolution of structure within real-world networks, and rely on observed properties and phenomenon to influence generation algorithms including internet traffic, biological interactions, and social networks \cite{aclpa,barabasi2012luck, barabasiPA, chung2003duplication, estrada2005complex}. 
In \cite{IIM} the second and fifth author introduced a new, Iterative Independent Model (IIM), generalizing previously defined models in \cite{ilm, ilt, ilat}. These models use a notion of cloning, inspired by social networks, where ``the friend of my friend is my friend'' and anticloning where ``the enemy of my enemy is my friend'' \cite{ilm, easley2010networks}. Specifically, given an initial graph $G$ and vertex $v \in V(G)$ we say a new vertex $v'$ is $v$'s clone if the neighborhood of $v'$ is exactly $v$ along with the neighborhood of $v$ and $v$'s anticlone if the neighborhood of $v'$ contains all vertices non-adjacent to $v$ (except $v$ itself). 
In the Iterative Independent Model (IIM) graphs are initiated with some starting graph $G = G_0$, and at each time step $t \geq 1$ built the graph $G_t$ from $G_{t-1}$ by either cloning or anticloning every vertex in $G_{t-1}$ independently \cite{IIM}. This model fully generalizes three previous models, including the Iterated Local Model (ILM) which itself generalizes, the Iterated Local Transitive (ILT) Model and the Iterated Local Anti-Transitive (ILAT) Model, all of which each used exclusively the cloning or anticloning procedures within each time step, respectively \cite{ilt, ilat}. 

Here we extend that line of work focused on information diffusion, specifically zero forcing and failed zero forcing, which we will formally introduce in Section \ref{prelim}. The zero forcing process has been widely studied since it was introduced independently in \cite{zf-phys} in 2007 and \cite{ZF1} in 2008, where it was given the name zero forcing. Interestingly, the motivation for zero forcing in \cite{zf-phys} was as a criterion for controllability of a quantum system (see also later work in \cite{quant2} and \cite{quantum}) while the motivation in \cite{ZF1} was to use the zero forcing number as a bound on a minimum rank of a graph. Since its introduction interest in the zero forcing number has spread including connections to spectral graph theory, quantum systems control, and social network analysis as well as interest in the parameter from a purely graph theoretic perspective. Additionally, failed zero forcing was first introduced in 2014 to study the other extreme of the zero forcing process \cite{failedzeroforcing}. When considering zero forcing as a model for information diffusion across a social network it is thus logical to investigate zero forcing on graph models which have been created to model social networks hence our focus on zero forcing and failed zero forcing on graphs created using the Iterated Local Transitive (ILT) Model and the Iterated Local Anti-Transitive (ILAT) Model. Here we establish initial results regarding zero forcing for the simpler models of ILT and ILAT graphs in hopes of laying the foundation for studying information diffusion more broadly on ILM and IIM graphs.

This paper is organized as follows. In Section \ref{prelim} we recall some basic graph notation as well as specific notation used throughout the later sections. In Section \ref{fzf-sec} we provide bounds on the failed zero forcing number of ILT and ILAT graphs. In Section \ref{zf-sec} we explore and provide bounds on the zero forcing number of ILT and ILAT graphs. Finally, in Section \ref{burn-sec} we conclude with a brief discussion of a related process, graph burning. 

\section{Preliminaries}\label{prelim}

First we recall notation and terminology used throughout the paper to avoid ambiguity. We will then formally state the definitions of the Iterated Local Transitive (ILT) and Iterated Local Anti-Transitive (ILAT) models. 

\subsection{Graph Notation}

Given a graph $G$ let $V(G)$ denote the vertex set of $G$. We use $N(v)$ to denote the neighborhood of $v \in V(G)$ and use the terms \emph{closed neighborhood of $v$} for $N[v]= \{v\} \cup N(v)$ and \emph{anti-neighborhood of $v$} for $V(G) \setminus N[v]$, `denoted $AN[v]$'. For any set $X \subseteq V(G)$ and $v \in V(G)$ we let $\deg_X(v)$ denote the number of neighbors of $v$ in $X$. 
The ILT and ILAT models are both complex networks generated by a deterministic process that grows the graph at each time step following a particular generation algorithm as defined below \cite{ilt, ilat}.

Given an initial graph $G$ and vertex $v \in V(G)$ we say a new vertex $v'$ is $v$'s clone if 
\[N(v')=N[v]
\]
and $v$'s anticlone if
\[N(v')=V(G)\setminus N[v] = AN[v].
\]

The ILT graphs are generated by initializing with some base graph $G = ILT_0(G)$. Here $ILT_1(G)$ is the graph obtained by simultaneously cloning every vertex in $G$ a single time. We then recursively define $ILT_l(G)=ILT_{1}(ILT_{l-1}(G))$. We refer to each cloning stage as a time step, and the clones added in time step $i$ are referred to as level $i$. The base graph will be referred to as level 0. Note that this process requires that each level, except level 0, is an independent set. ILAT graphs are generated in the same procedure as ILT graphs, but instead of cloning every vertex, we anticlone at every time step.
More recently in the literature these models have been generalized to the Iterated Independent Model by the second and fifth authors \cite{IIM}. In this generalization, each vertex is either cloned or anticloned independently, and in any single level could contain both clones and anticlones. In this model it was shown that the anticloning behaviors often dominate the graph structure. While Sections \ref{fzf-sec} and \ref{zf-sec} show that the zero forcing process behaves in significantly different ways for ILT and ILAT graphs in Section \ref{burn-sec} we see very clearly another instance where even just the existence of a single anticlone in the graph can fix the behavior of the graph burning process. 

Throughout this paper, we will need to refer to vertices and their predecessors and descendants with precision. If $v$ is a vertex in an ILT or ILAT graph then $v_i'$ will denote the clone/anticlone of $v$ added in level $i$ and $v^*$ will denote the pre-clone/pre-anticlone of $v$, i.e. if $v$ exists in level $j$ then $(v^*)_j'=v$. Note that our notation for the pre-clones/pre-anticlones suppresses the level as it will not always be known what level this vertex exists in. Often, we will nest these notations in order to accurately describe the descendants. For example, $(v'_i)'_j$ is the clone/anticlone of $v'_i$ in level $j$.
If $v$ is some vertex in an ILAT or ILT graph and $u$ is some vertex obtained from $v$ by successive cloning/anti-cloning we call $u$ a \emph{descendant} of $v$ and $v$ an \emph{ancestor} of $u$.

Finally, if $x$ is a descendant of $y$, and $y$ exists in level $i$, then let $c(x,y)$ denote the number of levels from level $i$, inclusive, and beyond which contain an ancestor of $x$. We will refer to $c(x,y)$ as the \emph{clone distance} between $x$ and $y$. Thus if $x$ is a clone of $y$ we would have $c(x,y)=1$.

\subsection{Zero Forcing}
The process of zero forcing on Graphs has been well studied; for a start see \cite{ZF1, zf-phys, bozeman2020multi, hogbenzeroforce, throttling}.
For a given graph $G$ the zero forcing process, first introduced independently in \cite{zf-phys} and \cite{ZF1} on $G$, initializes with the vertex set partitioned into a ``forced" set of vertices, $S$, and the remaining ``unforced" vertices\footnote{In traditional definitions of zero-forcing, you may see these nodes colored blue and white respectively. However, we choose to use the active descriptors of forced and unforced.}. At each time step, a forced vertex with exactly one unforced neighbor, $u$, will cause $u$ to become forced. We say $S$ is a zero forcing set if by iterating this process eventually every vertex becomes forced. The zero forcing number of $G$, denoted $Z(G)$, is the minimum cardinality of a zero forcing set.

Furthermore, it is also of interest to study which sets do not allow for a graph to become fully forced. More specifically, we focus on the notion of failed zero forcing which was first defined by Fetcie, Jacob, and Saavedra in \cite{failedzeroforcing}. Failed zero forcing, now considers subsets of initially forced vertices that, under the same forcing process as above, \emph{fail} to force every vertex in the graph. Here we call such sets failed zero forcing sets and let $FZ(G)$ denote the failed zero forcing number, the size of a \emph{largest} failed zero forcing set. This is the parameter we begin with in the following section.

\section{Failed Zero Forcing on ILT and ILAT Graphs}\label{fzf-sec}

Recall that the failed zero forcing number of a graph, is the largest size of a set that does not force the entire graph. We first address the case of ILAT graphs, which is arguably the more interesting case. Here, the failed zero forcing number is nearly independent of the base graphs. Namely, as seen in Theorem \ref{FZF-ILAT}, irrespective of the base graph, the failed zero forcing number is very large.

\begin{thm}\label{FZF-ILAT}
    For any base graph $G$, $FZ(ILAT_l(G)) \geq |V(ILAT_l(G))| - 6$ where $l \geq 5$.
\end{thm}
\begin{proof}
    Let $G$ be any graph and let $H = ILAT_5(G)$. We will prove that there is a set $U \subseteq V(H)$ of size 6 such that for every vertex $a \in V(H)$ we have $|U \cap N[a]\|$ and $|U \cap AN[a]|$ are both at least two. That is, every vertex in $H$ has at least two vertices in it's closed neighborhood and two non-neighbours in $U$. Thus for every $l \ge 5$ and every $v\in V(ILAT_l(G))$ we know that $N[v]$ and $AN[v]$ must each still contain at least two vertices in $U$. Therefore $V(ILAT_l(G))\setminus U$ is a failed zero forcing set of $ILAT_l(G)$ as desired. 
    
    To form $U$ let $u$ be an arbitrary vertex in $G$. Now in $H=ILAT_5(G)$ let $v=u'_1$ (recall this notation means $v$ is the anti-clone of $u$ in level $1$) and let $w = v'_2$ (the anticlone of v in level 2). We will select $U = \{v'_3, w'_3, (v'_3)'_4, (w'_3)'_4, ((v'_3)'_4)'_5, ((w'_3)'_4)'_5\}$. Note these are successive anticlones of $v$ and $w$ in levels 3, 4, and 5. 
     Let $a$ be an arbitrary vertex. We will prove that $|U \cap N[a]\|$ and $|U \cap AN[a]|$ are both at least two in cases based on the level $a$ is in. 
    
    First assume $a$ is in level 0. Note that all vertices in level 0 which are in $N_G[u]$  must be non-adjacent to $v$ and adjacent to $w$. Thus such vertices are adjacent to each vertex in $X=\{v'_3, (w'_3)'_4, ((v'_3)'_4)'_5\}$, and non-adjacent to the vertices in $U \setminus X$. Conversely, the vertices in $AN_G(u)$ are all adjacent to the vertices in $U \setminus X$ and non-adjacent to the vertices in $X$. Similarly all vertices in level $1$ except $v$ are adjacent to the vertices in $X$ and non-adjacent to the vertices in $U \setminus X$, while $v$ itself is adjacent to the vertices in $U \setminus X$ and non-adjacent to the vertices in $X$.

    In the case of vertices $a \neq w$ in level 2 we know $a$ is nonadjacent to $w$ and thus adjacent to $w'_3$ and $((w'_3)'_4)'_5$ and non-adjacent to $(w'_3)'_4$. Furthermore, if $a \nsim v$ then $a \sim v'_3$ and $a \nsim (v'_3)'_4$ and if $a \sim v$ then $a \sim (v'_3)'_4$ and $a \nsim v'_3$. On the other hand $w$ itself is non-adjacent to $v$ and thus $w$ is adjacent to $v'_3$ and $((v'_3)'_4)'_5$ and nonadjacent to $w'_3$ and $((w'_3)'_4)'_5$.

    Note that every vertex in levels $3, 4$, and $5$ of $H$ which are not in $U$ are non-adjacent to the two vertices in $U$ with which they share a level. Furthermore, $v'_3\nsim w'_3 $ and each are non-adjacent to their anticlones in level 4. Similarly, $(v'_3)'_4\nsim (w'_3)'_4 $ and each are nonadjacent to their anticlones in level 5. Finally, $((v'_3)'_4)'_5\nsim  ((w'_3)'_4)'_5$ and each are non-adjacent to their pre-anticlones in level 4. 

    Thus it remains to show that each vertex in levels $3, 4,$ and $5$ have at least two vertices in $U$ in their closed neighborhoods. As above every vertex in level $3$ which is not in $U$ is adjacent to $(v'_3)'_4$ and $(w'_3)'_4$, and every vertex in level $4$ which is not in $U$ is adjacent to $((v'_3)'_4)'_5$ and $((w'_3)'_4)'_5$. Furthermore, $v'_3 \sim (w'_3)'_4, ((v'_3)'_4)'_5 $, and  $w'_3 \sim (v'_3)'_4, ((w'_3)'_4)'_5 $. This immediately shows that every vertex in level 3, level 4, or $U$ has at least two elements of $U$ in its closed neighborhood. 

    We are now left to only show that each vertex $a$ in level $5$ which is not in $U$ has at least two vertices in $U$ in its closed neighborhood. Note that we have already shown that each vertex $b$ in levels $0$ through 2 are non-adjacent to at least two vertices from $\{v'_3, w'_3, (v'_3)'_4, (w'_3)'_4\}$. Thus $b'_5$ must be adjacent to at least two vertices from $\{v'_3, w'_3, (v'_3)'_4, (w'_3)'_4\}$. For each $b$ in level 3 except $v'_3$ and $w'_3$ we know $b'_5 \sim v'_3, w'_3$ and for each $b$ in level 4 except $(v'_3)'_4$ and $(w'_3)'_4$ we know $b'_5 \sim (v'_3)'_4, (w'_3)'_4$. Finally, the only vertices left to consider are $(v'_3)'_5$ and $(w'_3)'_5$. Note that $(v'_3)'_5 \sim w'_3, (v'_3)'_4$ and $(w'_3)'_5 \sim v'_3, (w'_3)'_4$, proving our result. 
\end{proof}
Given Theorem \ref{FZF-ILAT} we now begin the process of classifying exactly what base graphs lead to exactly what zero forcing number of the corresponding ILAT graphs. We are able to completely categorize when $FZ(ILAT_l(G)) = |V(ILAT_l(G))| - 2$ and when $FZ(ILAT_l(G)) = |V(ILAT_l(G))| - 3$ as seen in Fact \ref{FZF-ILAT2} and Theorem \ref{FZF-ILAT3}, respectively, and give some initial thoughts on when we get $FZ(ILAT_l(G)) = |V(ILAT_l(G))| - 4$.

\begin{fact}\label{FZF-ILAT2}
To have $FZ(ILAT_l(G))=|V(ILAT_l(G))|-2$, for $l \ge 1$ one of the following must hold
        \begin{enumerate}
        \item Two vertices in our base graph have identical closed neighborhoods.
        \item The base graph, $G$, consists of an isolated vertex, $x$, and a vertex which dominates $G \setminus \{x\}$. 
        \item $G=K_1$
            
        \end{enumerate}
\end{fact}
\begin{proof}
    It is clear that for any graph $G$ the only way for $FZ(G)= |V(G)|-2$ is for there to exist two vertices $u, v \in V(G)$ such that $N(u)\setminus\{v\}=N(v) \setminus \{u\}$, thus defining case 1. Then $V(G)\setminus \{u,v\}$ is clearly a failed zero forcing set. Since every vertex is adjacent to either both unforced vertices or neither, there is no one to force these remaining vertices. 
    If the base graph is in case 2, then $G$ consists of an isolated vertex, $u$, and a vertex, $v$, which dominates $G \setminus \{u\}$. After the first step in $ILAT_1(G)$ we have $N[u_1']=N[v]$, and thus are in case 1.
    Finally, if the base graph is $K_1$, with $V(G) = \{v\}$, then after one time step we will enter case 2 with the isolated vertex $v$, and the vertex that dominates $G\backslash \{v\}$ being $v'$ its anti-clone, landing us in case 1. 
    Thus each of the three conditions in Fact \ref{FZF-ILAT2} are clearly sufficient and it just remains to show they are necessary.

    First note that every base graph $G$ on one or two vertices clearly satisfies one of the conditions of Fact \ref{FZF-ILAT2}, so we may assume $G$ has at least three vertices. Now assume that for some level $l \ge 1$ we have $FZ(ILAT_l(G))=|V(ILAT_l(G))|-2$. Let $l^*$ be the smallest such $l$ and let $H= ILAT_{l^*}(G)$. As earlier let $H_{l^*}$ denote the final level of $H$. Let $F$ denote the subgraph of $H$ induced by all previous levels, $V(H)\setminus H_{l^*}$. Since we assume $FZ(H)= |V(H)|-2$ we must have two vertices $u, v \in V(H)$ such that $N_H(u)\setminus\{v\}=N_H(v) \setminus \{u\}$. 
    
    Consider the case where $u, v \in V(F)$. By the minimality of $l^*$ we may assume, without loss of generality, that there is an $x \in V(F)$ such that $x \in N(u) \setminus N(v)$, a contradiction. Now consider the case where $u, v \in H_{l^*}$. If we consider their pre-anticlones, $u^*$ and $v^*$ we are thus forced to have $AN_F[u^*]=AN_F[v^*]$. However, this immediately gives that $N_F[u^*]=N_F[v^*]$ again contradicting the minimality of $l^*$. Thus we are only left to consider the case where $u \in V(F)$ and $v \in H_{l^*}$. This requires that for every $x \in V(F)\setminus \{u,v^*\}$ we have $x \sim u$ or else $x'_l \in N(u) \setminus N(v)$. Since $G$ has at least three vertices every level contains at least three vertices, and thus if $u \in H_{i}$ for $i\ge 1$ it has at least two non-neighbors in $F$, namely all of the other vertices in $H_i$. 
    
    Thus we are only left to consider the situation where $u$ is in level 0, i.e. $u \in V(G)$. However, in this case if $l^* \ge 2$ then $u$ cannot be adjacent to both a vertex in $G$ and that vertices anticlone in level $1$. Thus we must have $l^*=1$. Here we note that $u \nsim v^*$ since clearly $v^* \nsim v$ and we have already shown that $u$ must dominate $V(G)\setminus{v^*}$, which in turn requires that $v$ must also be adjacent to each vertex in $V(G)\setminus{v^*}$. Thus we are exactly in the configuration described in Case 2. 
\end{proof}
Note that Fact \ref{FZF-ILAT2} completely classifies all cases where $FZ(ILAT_l(G))= |V(ILAT_l(G))|-2$ for all $l$. 
However, perhaps surprisingly, it is impossible to have a base graph such that $FZ(ILAT_l(G))= |V(ILAT_l(G))|-3$ for all $l$. In fact, we prove an even stronger statement below.

\begin{thm}\label{FZF-ILAT3}
    For any graph $G$ and $l \geq 4$ we have $FZ(ILAT_l(G)) \neq |V(ILAT_l(G))| - 3$.
\end{thm}
\begin{proof}
    As noted earlier any base graph $G$ on at most two vertices satisfies one of the conditions from Fact \ref{FZF-ILAT2}. Thus we may assume $|V(G)|\ge 3$ and $l \ge 4$. To show that any such $H=ILAT_l(G)$ cannot have $FZ(H) = |V(H)| - 3$ we will demonstrate that for any triple of vertices $a_1, a_2, a_3 \in V(H)$ we have one of the below conditions.
    \begin{enumerate}
        \item There is a vertex $x \in V(H)\setminus \{a_1,a_2,a_3\}$ such that $N(x)$ contains exactly one of $a_1,a_2, a_3$. \label{-3con1}
        \item There are vertices $u,w \in V(H)$ such that $N[u]=N[w]$.\label{-3con2}.
    \end{enumerate}

    Condition \ref{-3con1} shows that that $V(H)\setminus \{a_1,a_2,a_3\}$ is not a maximum failed zero forcing set. 
    Similarly, condition \ref{-3con2} shows that $FZ(H)=|V(H)|-2$, by excluding $u,w$. 
    
    Below we divide the proof into several cases based on which level each of the $a_i's$ appear. In each case, we identify either a vertex $x$ satisfying condition \ref{-3con1} or two vertices $u$ and $w$ having identical closed neighborhoods, satisfying condition \ref{-3con2}. As before let $H_i$ denote the vertices of $H$ in level $i$. Let $A=\{a_1,a_2,a_2\}$.

    \medskip

    \noindent {\bf Case 1:} $A \subset  \bigcup_{i=0}^{l-2}H_{i}$\\
    If one of the vertices in $A$, without loss of generality assume $a_1$, is not adjacent to the other two vertices in $A$ then we can satisfy condition \ref{-3con1} with $x=((a_1)_{l-1}')_l'$. Now assume there is a vertex in $A$ adjacent to exactly one of the other two. Without loss of generality assume $a_1 \sim a_2$ and $a_1 \nsim a_3$ then we can satisfy condition \ref{-3con1} with $(a_1)'_{l}$. Finally assume that $A$ induces a complete graph. If there exists some vertex $x \in \bigcup_{i=0}^{l-2}H_{i} $ such that $x$ is adjacent to exactly one vertex in $A$ then condition \ref{-3con1} is satisfied, and if there exists a vertex $v \in  \bigcup_{i=0}^{l-2}H_{i} $ such that $v$ is adjacent to exactly two vertices in $A$ then condition \ref{-3con1} is still satisfied by $x=v'_l$. If neither of these is the case then the vertices in $A$ all have identical closed neighborhoods satisfying condition \ref{-3con2}.\\

     \noindent {\bf Case 2:} $|A \cap  \bigcup_{i=0}^{l-2}H_{i}|=2$ and $|A \cap H_{l-1}|=1$.\medskip

Without loss of generality let $a_1 \in H_{l-1}$ and $a_2,a_3 \in \bigcup_{i=0}^{l-2}H_{i} $. If one of the vertices in $\{a_2,a_3\}$, without loss of generality assume $a_2$, is not adjacent to the other two vertices in $A$ then we can satisfy condition \ref{-3con1} with $x=((a_2)_{l-1}')_l'$. If a vertex in $A$, without loss of generality $a_2$, is adjacent to exactly one other vertex in $A$ then condition \ref{-3con1} is satisfied with $x=(a_2)'_l$. We are again left to consider the case when the vertices in $A$ induce a complete graph. As before if there is a vertex $v \in  \bigcup_{i=0}^{l-2}H_{i}\setminus A$ such that $v$ is adjacent to exactly one of $\{a_2,a_3\}$ then condition \ref{-3con1} is satisfied with $x=v'_{l-1}$. If no such $v$ exists then $a_2$ and $a_3$ must have identical closed neighborhoods satisfying condition \ref{-3con2}.\\

     \noindent {\bf Case 3:} $|A \cap \bigcup_{i=0}^{l-2}H_{i}|=1$ and $|A \cap H_{l-1}|=2$.
     \medskip

Without loss of generality let $a_1,a_2 \in H_{l-1}$ and $a_3 \in\bigcup_{i=0}^{l-2}H_{i} $. Note since $a_1$ and $a_2$ are in the same level (and not both in level 0) we know $a_1 \nsim a_2$. If at least one of $\{a_1,a_2\}$ is adjacent to $a_3$, without loss of generality assume $a_1 \sim a_3$, then condition \ref{-3con1} is satisfied with $x=(a_1)_l'$. If not, $A$ forms an independent set. Note that since $l-2 \ge 1$ we know that there is some vertex $v \in \bigcup_{i=0}^{l-2}H_{i}$ such that $v \nsim a_3$. Note that $v'_{l-1} \sim a_3$ and thus cannot be in $A$. Therefore condition \ref{-3con1} is satisfied with $x=v'_{l-1}$.\\

     \noindent {\bf Case 4:} $|A \cap \bigcup_{i=0}^{l-2}H_{i}|=2$ and $|A \cap H_{l}|=1$.
     \medskip
     
     Without loss of generality let $a_1 \in H_l$ and $a_2, a_3 \in \bigcup_{i=0}^{l-2}H_{i}$. If $a_2 \nsim a_3$ then at most one of $(a_2)_l'$ and $(a_3)_l'$ is equal to $a_1$ and thus the other satisfies condition \ref{-3con1}. If there exists some $v \in \bigcup_{i=0}^{l-2}H_{i}\setminus A$ such that $v$ is adjacent to exactly one of $\{a_2,a_3\}$ then $v'_{l-1}$ is also adjacent to exactly one of $\{a_2,a_3\}$. Therefore at most one of $v'_l$ and $(v'_{l-1})'_l$ is equal to $a_1$ and thus the other satisfies condition \ref{-3con1}. If no such $v$ exists then $N[a_2]=N[a_3]$ satisfying condition \ref{-3con2}.\\
     
       \noindent {\bf Case 5:} $|A \cap \bigcup_{i=0}^{l-2}H_{i}|=1$, $|A \cap H_{l-1}|=1$, and $|A \cap H_{l}|=1$.
       
       \medskip
       
     Without loss of generality let $a_1 \in H_l$, $a_2 \in H_{l-1}$, and $a_3 \in \bigcup_{i=0}^{l-2}H_{i}$.
    Note that since $l-2 \ge 1$ we know there are at least three vertices in $ \bigcup_{i=0}^{l-2}H_{i}\setminus A$ which are not adjacent to $a_3$. Thus there is a vertex $v \in \bigcup_{i=0}^{l-2}H_{i} \setminus A$ such that $v \nsim a_3$, $v_{l-1}'\neq a_2$, and $(v_{l-1}')'_l\neq a_1$. As vertices in the same non-zero level are non-adjacent we know that condition \ref{-3con1} is satisfied with $x=(v_{l-1}')'_l$.\\
    
         \noindent {\bf Case 6:} $|A \cap \bigcup_{i=0}^{l-2}H_{i}|=1$ and $|A \cap H_{l}|=2$.

         \medskip
         
     Without loss of generality let $a_1,a_2 \in H_l$ and $a_3 \in \bigcup_{i=0}^{l-2}H_{i}$.
    Again we know there are at least three vertices in $ \bigcup_{i=0}^{l-1}H_{i}\setminus A$ which are not adjacent to $a_3$. Thus there is a vertex $v \in  \bigcup_{i=0}^{l-1}H_{i} \setminus A$ such that $v\nsim a_3$, and $v'_l$ is neither $a_1$ nor $a_2$. Therefore, condition \ref{-3con1} is satisfied with $x=v'_l$.\\
    
    \noindent {\bf Case 7:} $|A \cap H_{l-1}|=3$.

    \medskip 
    
 Note that in this case if there is some vertex $v \in   \bigcup_{i=0}^{l-2}H_{i}$ such that $v$ is adjacent to exactly one vertex in $A$ then condition \ref{-3con1} is automatically satisfied with $x=v$. Additionally, if there is some vertex $v \in  \bigcup_{i=0}^{l-2}H_{i}$ such that $v$ is adjacent to exactly two vertices in $A$ then condition \ref{-3con1} is satisfied by $x=v'_l$. Thus we may assume every vertex in $\bigcup_{i=0}^{l-2}H_{i}$ is adjacent to all of the vertices in $A$ or none of them. In particular, the pre-anticlones of the vertices in $A$, denoted $A^* = \{a_1^*,a_2^*, a_3^*\}$, must have identical \emph{closed} neighborhoods in $H$ satisfying condition \ref{-3con2}. \\
 
\noindent {\bf Case 8:} $|A \cap H_{l-1}| \geq 1$ and $A\backslash H_{l-1} \subset H_l$ 
\medskip

Without loss of generality $a_1 \in H_{l-1}$. Since $l-1 \ge 2$ we know there are at least three other vertices in $H_{l-1}$, all of which are non-adjacent to $a_1$. Thus there is some vertex $v \in H_{l-1}\setminus A$ such that $v'_l \notin A$. Therefore condition \ref{-3con1} is satisfied by $x = v'_l$. \\

\noindent {\bf Case 9:} $|A \cap H_{l}|=3$

\medskip 

This case requires the most work as we cannot take advantage of any anticlones of $A$. Thus we must instead consider the pre-anticlones again denoted by $A^* = \{a_1^*,a_2^*, a_3^*\}$. Please note that there are no a priori guarantees on the levels of the vertices in $A^*$ and we will need to divide into further cases based on what levels these vertices are in. Note first that if any vertex $v \in  \bigcup_{i=0}^{l-1}H_{i}$ has exactly two vertices of $A^*$ in its closed neighborhood then condition \ref{-3con1} is satisfied with $x=v$. In particular, this means we may assume that $A^*$ itself forms either an independent set or induces a $K_3$.

\medskip

\indent{\bf Case 9a:} $A^* \subset  \bigcup_{i=0}^{l-2}H_{i}$

\smallskip 
First consider the case where $A^*$ forms an independent set. Thus $x= (a_1^*)'_{l-1}$ satisfies condition \ref{-3con1}. The only other case now is when $A^*$ induces a $K_3$. As noted above we are done if any vertex $v \in \bigcup_{i=0}^{l-1}H_{i}$ has exactly two vertices in $A^*$ in its closed neighborhood. However, if any vertex $v \in \bigcup_{i=0}^{l-2}H_{i}$ has exactly one vertex in $A^*$ in its closed neighborhood then $v_{l-1}'$ will have exactly two vertices of $A^*$ in its closed neighborhood. Thus we may assume that every vertex $v \in \bigcup_{i=0}^{l-1}H_{i}$ has either all of $A^*$ or none of $A^*$ in its closed neighborhood. Thus the vertices in $A^*$ have identical closed neighborhoods satisfying condition \ref{-3con2}.

\medskip

\indent{\bf Case 9b:} $|A^* \cap  \bigcup_{i=0}^{l-2}H_{i}|=2$ and $|A^* \cap H_{l-1}|=1$.

\smallskip

Without loss of generality let $a_1^* \in H_{l-1}$. Again we start by assuming $A^*$ forms an independent set. In this case let $b$ be the pre-anticlone of $a_1^*$, that is $a_1 = ((b)_{l-1}')_l'$. As we are assuming $A^*$ is an independent set we know $a^*_2,a^*_3 \in N(b)$. Thus condition \ref{-3con1} is satisfied by $x=b$. On the other hand if $A^*$ induced a $K_3$ then $a^*_2$ and $a^*_3$ must either be in different levels or both in level 0. First assume $a^*_2$ and $a^*_3$ are both in level 0. Then at most one of $((a^*_2)_{l-2}')_{l-1}'$ and $((a^*_3)_{l-2}')_{l-1}'$ is equal to $a_1^*$; as they all lie in the same non-zero level, the other is non-adjacent to $a_1^*$ thus satisfying condition \ref{-3con1}. Now assume $a^*_2$ and $a^*_3$ are in different levels, and, without loss of generality that $a^*_2$ is in $H_{k}$ and $a^*_3$ is in $H_{j}$ with $j<k$.
Note for any $w \in ( \bigcup_{i=0}^{k-1}H_{i})\cap N[a^*_3]$ we have that $w_k'$ is nonadjacent to both $a^*_3$ and $a_2^*$. Thus if $(w_{k}')'_{l-1}\neq a_1^*$ then condition \ref{-3con1} is satisfied by $x=(w_{k}')'_{l-1}$. If no such $w$ exists this means $a_3^*$ is an isolate in the graph induced by $\bigcup_{i=0}^{k-1}H_{i}$ and $((a_3^*)_{k}')'_{l-1}= a_1^*$. Due to the anti-cloning process such an isolate could only exist if one of the following holds: 
\begin{enumerate}
    \item $k=1$ and $a_3^*$ is an isolate in the base graph
    \item $k=2$ and $j=1$ and the pre-anticlone of $a_3^*$, which must lie in level 0, dominates level 0.
\end{enumerate}
For the first of the above cases note that $(a_3^*)_2'$ is also nonadjacent to both $a_3^*$ and $a_2^*$ and $((a_3^*)_2')'_{l-1}\notin N[a_1^*]$, thus $x = ((a_3^*)_2')'_{l-1}$ satisfies condition \ref{-3con1} as desired . In the second case if there is some $c$ in level 0 such that $c \nsim a_2^*$ then again we would have $c'_{l-1}$ is adjacent to both $a_3^*$ and $a_2^*$ and is not in $N[a_1*]$. Thus $x=c'_{l-1}$ satisfies condition \ref{-3con1}. If such a $c$ does not exist then $a_2^*$ must be an anticlone of a level 1 vertex $d$ such that $d\neq a_3^*$ (as $a_3^* \sim a_2^*$) and $d$ is also an isolate in $H_0\cup H_1$. Thus both the pre-anticlones of $d$ and $a_3^*$ dominate $H_0$ and thus have identical closed neighborhoods satisfying condition \ref{-3con2}.

\medskip

\indent{\bf Case 9c:} $|A^* \cap \bigcup_{i=0}^{l-2}H_{i}|=1$ and $|A^* \cap H_{l-1}|=2$.

\smallskip

Without loss of generality let $a_1^*, a_2^* \in H_{l-1}$,
Note that since $a_1^*$ and $a_2^*$ are both in level $l-1$ we know $a_1^* \nsim a_2^*$. Thus $a_1^* \sim a_2$ and $a_2^* \sim a_1$. Therefore if either $a_1^*$ or $a_2^*$ is non-adjacent to $a_3$ it must satisfy condition \ref{-3con1}. Thus assume $a_1^*,a_2^* \nsim a_3^*$, and $(a_1^*)^*,(a_2^*)^* \sim a_3^*$. This immediately implies $(a_1^*)^*,(a_2^*)^* \nsim a_3$ and clearly $(a_1^*)^* \sim a_1$ and $(a_2^*)^*\sim a_2$. Hence, if $(a_1^*)^*\nsim (a_2^*)^*$ then $(a_1^*)^* \nsim a_2$ so $x=(a_1^*)^*$ satisfies condition \ref{-3con1}. Thus we may assume $\{(a_1^*)^*, (a_2^*)^*, a^*_3\}$ must induce a $K_3$.

    Before proceeding, we must now also eliminate the possibility that $\{(a_1^*)^*, (a_2^*)^*, a_3^*\}$ all lie in level 0. To do this simply note that $(a_1^*)^*, (a_2^*)^*, a^*_3 \nsim [(a_1^*)^*]'_1 $. Thus $a_1, a_2 \nsim [(a_1^*)^*]'_1$, but $a_3 \sim [(a_1^*)^*]'_1$ so $x=[(a_1^*)^*]'_1$ satisfies condition \ref{-3con1}. Now assume $(a_1^*)^*$ lives in level $i$, $(a_2^*)^*$ lives in level $j$, and $a_3^*$ lives in level $k$ for some $i,j,k$. We know by our above argument that $\max\{i,j,k\}\ge 1$. We then further divide the argument based on which of $i, j,$ or $k$ is largest. First consider when $k = \max\{i,j,k\}$. Recall that since $\{(a_1^*)^*, (a_2^*)^*, a^*_3\}$ must induce a clique and $k \ge 1$ we know $i,j <k$. Thus $[(a_1^*)^*]'_k \nsim a^*_3, (a_1^*)^*, (a_2^*)^*$ implying $[(a_1^*)^*]'_k \sim a_3$ but $[(a_1^*)^*]'_k \nsim a_1, a_2$. Thus $x=[(a_1^*)^*]'_k$ satisfies \ref{-3con1}. On the other hand consider when either $i$ or $j$ is $\max\{i,j,k\}$, without loss of generality assume $j =\max\{i,j,k\}$. In this case we again have  $[(a_1^*)^*]'_j \nsim a^*_3, (a_1^*)^*, (a_2^*)^*$ and thus $x=[(a_1^*)^*]'_k$ satisfies condition \ref{-3con1}.

    \medskip
    
    \indent{\bf Case 9d:} $A^* \subset H_{l-1}$.

    \smallskip
    
    Unfortunately, this sub-case will need to be broken down further. First consider when all of $(a_1^*)^*$, $(a_2^*)^*$, and $(a_3^*)^*$ are in level 0. Note that if there is a vertex $w\in V(G)$ that has exactly one of $(a_1^*)^*$, $(a_2^*)^*$, or $(a_3^*)^*$ in its closed neighborhood then $x=w$ satisfies condition \ref{-3con1}. Similarly, if there is a vertex $w\in V(G)$ that has exactly two of $(a_1^*)^*$, $(a_2^*)^*$, or $(a_3^*)^*$ in its closed neighborhood, without loss of generality $((a_1)^*)^*$ and $((a_2)^*)^*$, then $w'_1 \sim ((a_3)^*)^*$ and $w'_1 \nsim ((a_1)^*)^*, ((a_2)^*)^*$ and thus $x=w'_1$ satisfies condition \ref{-3con1}. Therefore each vertex in the base graph must have either none of $(a_1^*)^*$, $(a_2^*)^*$, and $(a_3^*)^*$ or all of $(a_1^*)^*$, $(a_2^*)^*$, and $(a_3^*)^*$ in its closed neighborhood. Thus all the vertices in $\{(a_1^*)^*, (a_2^*)^*, (a_3^*)^*\}$ have identical closed neighborhoods satisfying condition \ref{-3con2}. 
    
    Now we may assume that at least one of $(a_1^*)^*$, $(a_2^*)^*$, or $(a_3^*)^*$ is not in level 0. Recall that we may also assume that $|V(G)| \ge 3$. Additionally, note that if any one of $(a_1^*)^*$, $(a_2^*)^*$, or $(a_3^*)^*$ is non-adjacent to the other two then it satisfies condition \ref{-3con1}. In particular, this means we know that not all of $(a_1^*)^*$, $(a_2^*)^*$, and $(a_3^*)^*$ lie in the same level. 

    Now we will handle the case where exactly two of $(a_1^*)^*$, $(a_2^*)^*$, and $(a_3^*)^*$ lie in the same level. Again this must be subdivided into yet more subcases. Note that we may assume, without loss of generality, that the two vertices in the same level are $(a_1^*)^*$ and $(a_2^*)^*$. Thus we will start by considering the case where $(a_3^*)^*$ lives in level $j$ which, we assume in this case, is after the level containing $(a_1^*)^*$ and $(a_2^*)^*$. If $(a_1^*)^* \nsim (a_2^*)^*$ then we know $(a_1^*)^*, (a_2^*)^* \sim (a_3^*)^*$ as each of these three vertices must have another one of the three in its closed neighborhood, or else it would satisfy condition \ref{-3con1}. Here we simply note that $((a_1^*)^*)'_j \sim (a_2^*)^*$ and not adjacent to both $(a_1^*)^*$ and $(a_3^*)^*$. Thus $x=((a_1^*)^*)'_j$ satisfies condition \ref{-3con1}. On the other hand, if $(a_1^*)^* \sim (a_2^*)^*$ we first know that $(a_1^*)^*$ and $(a_2^*)^*$ must live in level 0 and that there is at least one more vertex in level 0. If there is a vertex $w \in H_0$ adjacent to exactly one of $(a_1^*)^*$ and $(a_2^*)^*$, say $(a_1^*)^*$, then $w'_j$ is adjacent to $(a_2^*)^*$ and non-adjacent to both $(a_1^*)^*$ and $(a_3^*)^*$. Thus $x=w'_j$ satisfies condition \ref{-3con1}.
    Otherwise, for every $w \in H_0 \{(a_1^*)^*,(a_2^*)^*\}$ we know $w$ is adjacent to both or neither of $(a_1^*)^*$ and $(a_2^*)^*$, so these two vertices have identical closed neighborhoods, satisfying condition \ref{-3con2}.   
    
    To finish the case where exactly two of $(a_1^*)^*$, $(a_2^*)^*$, and $(a_3^*)^*$ lie in the same level we must also consider when $(a_3^*)^*$ lives in level $j$, a level before the level containing $(a_1^*)^*$ and $(a_2^*)^*$. Note that here we must have $(a_1^*)^*\nsim (a_2^*)^*$. Thus $(a_3^*)^* \sim (a_1^*)^*, (a_2^*)^*$ or else one of $(a_1^*)^*$ and $x=(a_2^*)^*$ would satisfy condition \ref{-3con1}. This is the one moment where we will make full use of our assumption that $l \ge 4$. Here we will let level $i$ be the level of $(a_1^*)^*$ and $(a_2^*)^*$, so $j<i$. Note $j<i \le l-2$ and $l \ge 4$ there must be another level before level $l-1$ which is not level $i$ or level $j$. Call this level $k$. If $k>i$ then consider $((a_2^*)^*)'_k$ which we know will be adjacent to $(a_1^*)^*$ and non-adjacent to $(a_2^*)^*$ and $(a_3^*)^*$. Thus $x=((a_2^*)^*)'_k$ satisfies condition \ref{-3con1}. Now consider when $k<i$ then since $|H_0|\ge 3$ we know $(a_3^*)^*$ must have at least three non-neighbors which live in a level prior to level $i$. To see this note that if $i\neq 0$ then $(a_3^*)^*$ is non-adjacent to it's own pre-anticlone as well as the other vertices in it's level, of which there are at least two. Similarly if $i = 0$ then for each $w\in H_0$, of which there are at least three, $(a_3^*)^*$ is non-adjacent to exactly one of $\{w,w_1'\}$. Therefore, since there are at least three non-neighbors of $(a_3^*)^*$ appearing in a level prior to level $i$ we know there is at least one non-neighbor $u$ of $(a_3^*)^*$ appearing prior to level $i$ which is not itself the pre-anticlone of $(a_1^*)^*$ or $(a_2^*)^*$. Thus $u'_i \neq (a_1^*)^*, (a_2^*)^*$ so $u'_i$ is non-adjacent to $(a_1^*)^*, (a_2^*)^*$ and adjacent to $(a_3^*)^*$. Thus $x=u'_i$ satisfies condition \ref{-3con1}. 

    The final case we have to consider is when $(a_1^*)^*$, $(a_2^*)^*$, and $(a_3^*)^*$ all lie in distinct levels. Say $(a_1^*)^*$ lies in level $i$, $(a_2^*)^*$ lies in level $j$, and $(a_3^*)^*$ lies in level $k$ where, without loss of generality $i<j<k$. 
    First consider that there is at least one vertex $w\in \bigcup_{m=0}^{j-1}H_m \setminus \{((a_2^*)^*)^*\}$ non-adjacent to $(a_1^*)^*$. Then we know $w'_j$ is adjacent to $ (a_1^*)^*$ and non-adjacent to $(a_2^*)^*$. Thus $(w'_j)'_k$ is adjacent to $(a_2^*)^*$ and non-adjacent to both $(a_1^*)^*$ and $(a_3^*)^*$. Therefore $x=(w'_j)'_k$ satisfies condition \ref{-3con1}. If no such $w$ exists then $(a_1^*)^*$ must be adjacent to all but one vertex in $\bigcup_{m=0}^{j-1}H_m$. This is only possible if $i=0$ and $j=1$, as if $i \neq 0$ then $(a_1^*)^*$ is non-adjacent to all the other vertices in it's level, of which there are at least two, and if $i=0$ but $j>1$ then each of the vertices $b\in V(G)\setminus\{(a_1^*)^*\}$, of which there are at least two, we know $(a_1^*)^*$ must be adjacent to exactly one of $b, b'_1$. Furthermore, among the double pre-anticlones if $(a_1^*)^* \nsim (a_2^*)^*, (a_3^*)^*$ then $x=(a_1^*)^*$ satisfies condition \ref{-3con1}. On the other hand if $(a_1^*)^* \nsim (a_2^*)^*$, but $(a_1^*)^* \sim (a_3^*)^*$ then $x=((a_1^*)^*)_k'$ satisfies condition \ref{-3con1}. Therefore, we may now assume $(a_1^*)^* \sim (a_2^*)^*$, which would require that $(a_1^*)^* \nsim ((a_2^*)^*)^*$ Thus we know that $(a_1^*)^* \in H_0$ and is adjacent every vertex in $V(G)\setminus \{(a_1^*)^*, ((a_2^*)^* )^* \}$. If there exist some $c \in H_0\setminus \{(a_1^*)^*, ((a_2^*)^* )^* \}$ which is adjacent to $((a_2^*)^* )^*$ and $c'_k \neq (a_3^*)^*$ then $c_k'$ is adjacent to $(a_2^*)^*$ and non-adjacent to both $(a_1^*)^*$ and $(a_3^*)^*$. Thus $c'_k$ satisfies condition \ref{-3con1}. If $c_k' = (a_3^*)^*$ then $c$ itself satisfies condition \ref{-3con1}. Thus we know $((a_2^*)^* )^*$ must be an isolated vertex in $G$. Therefore, in $H_0 \cup H_1$ we know $(a_1^*)^* $ and $(a_2^*)^*$ must have identical closed neighborhoods satisfying condition \ref{-3con2}.

\end{proof}
While we do not have a complete classification for when ILAT graphs have failed zero forcing number $FZ(ILAT_l(G))= |V(ILAT_l(G))|-4$ for all $l$ and $FZ(ILAT_l(G))= |V(ILAT_l(G))|-5$ for all $l$ below is an infinite family of graphs which, for example, have $FZ(ILAT_l(G))= |V(ILAT_l(G))|-4$ for all $l$. 
\begin{fact}\label{FZF-ILAT4}
    Let $G$ be a graph with vertices $u$ and $v$ such that $N(u)= G\setminus \{v\}$ and $N(v)= G\setminus \{u\}$, then $FZ(ILAT_l(G))=|V(ILAT_l(G))|-4$ for $l\ge 1$.
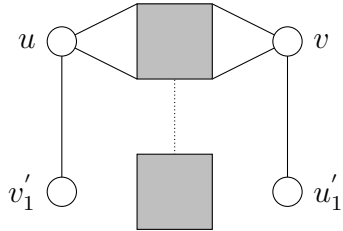
\begin{figure}[h]
    \centering
    \begin{tikzpicture}[scale=.5]
    \draw[color=black] (0,-4) -- (0,0);
    \draw[color=black] (6,-4) -- (6,0);
    \draw[densely dotted,color=black] (3,-4) -- (3,0);
    \draw[color=black] (0,0) -- (2,1);
    \draw[color=black] (0,0) -- (2,-1);
    \draw[color=black] (6,0) -- (4,1);
    \draw[color=black] (6,0) -- (4,-1);
    \fill[gray!50] (2,-1) rectangle ++(2,2);
    \draw[draw=black] (2,-1) rectangle ++(2,2);
    
    \draw (0,0) node [circle,draw,fill=white,label= west:$u$]{ };
    \draw (6,0) node [circle,draw,fill=white,label= east:$v$]{ };

    \fill[gray!50] (2,-5) rectangle ++(2,2);
    \draw[draw=black] (2,-5) rectangle ++(2,2);
    
    \draw (0,-4) node [circle,draw,fill=white,label= west:$v_1^{'}$]{ };
    \draw (6,-4) node [circle,draw,fill=white,,label= east:$u_1^{'}$]{ };
\end{tikzpicture}  
    \caption{A configuration for $FZ(ILAT_l(G))=|V(ILAT_l(G))|-4$ for $l\ge 1$.}
    \label{fig:2 dominating-ish_fig}
\end{figure}

\begin{proof}
    Let $G$ be a graph with vertices $u$ and $v$ such that $N(u)= G\setminus \{v\}$ and $N(v)= G\setminus \{u\}$. Now consider $H=ILAT_1(G)$. Note that $N_H(u_1')=\{v\}$ and $N_H(v_1')=\{u\}$. We claim that for any such graph $G$ and any $l \ge 1$ we have $V(ILAT_l(G))\setminus \{u,v,u_1',v_1'\}$ is a failed zero forcing set of $ILAT_l(G)$. Note that in $H$ every vertex in level 0 contains exactly two vertices from $\{u,v,u_1',v_1'\}$ in it's closed neighborhood. Furthermore, in level 1 all vertices, except $u_1'$ and $v_1'$ are adjacent to none of the vertices in $\{u,v,u_1',v_1'\}$. Finally $u_1'$ and $v_1'$ each have exactly two vertices from $\{u,v,u_1',v_1'\}$ in their closed neighborhood. Thus every vertex in $ILAT_l(G)$ for $l \ge 2$ will have $0, 2, $ or $4$ vertices from $\{u,v,u_1',v_1'\}$ in their closed neighborhood. Thus $V(ILAT_l(G))\setminus \{u,v,u_1',v_1'\}$ is a failed zero forcing set as desired. 

    We know this failed zero forcing set is of maximum size since by Theorem \ref{FZF-ILAT3} there can never be a maximum failed zero forcing set of size $|V(ILAT_l(G))|-3$ and by Fact \ref{FZF-ILAT2} we know such graphs will not have a failed zero forcing set of size $|V(ILAT_l(G))|-2$. Thus $FZ(ILAT_l(G))=|V(ILAT_l(G))|-4$.
\end{proof}
\end{fact}
As is typical for ILT graphs their behavior, unlike for ILAT graphs, is often much more dependent on the structure of the base graph. This again comes into play for failed zero forcing on ILT graphs. To see this we need to introduce a related concept of loop zero forcing \cite{loopzero}. In loop zero forcing we retain the standard zero forcing rules and add in one additional method of forcing. Namely, an unforced vertex can become (loop) forced if all of it's neighbors are forced. Thus it is simple to check that if $W$ is a failed \emph{loop} zero forcing set on some graph $G$ then in $H=ILT_l(G)$ we have $W \cup \left(\bigcup_{i=1}^l H_i\right)$ is also a failed zero forcing set in $H$ (recall $H_i$ denotes the vertices in level $i$ of $H$).

\section{Zero Forcing on ILT and ILAT Graphs}\label{zf-sec}
In this section we turn our attention to providing bounds for the zero forcing number of both ILT and ILAT graphs. Of note, for Theorem \ref{cloneupper} which concerns ILT graphs, the upper bound is dependent on the zero forcing number of the base graph. However, in the case of ILAT graphs the bounds given in Theorem \ref{anticlonezero} are independent of the base graph. As the cloning process maintains the behavior of the base graph significantly more than anti-cloning it is conceivable that all zero forcing numbers of ILAT graphs on a fixed number of vertices are contained in a much smaller interval than in the ILT case. 
\begin{thm}\label{cloneupper}
For any graph $G$ and $H = ILT_l(G)$ we have that a zero forcing set of $G$ along with all its descendants is a zero forcing set of $H$. Thus $Z(H)\le 2^l Z(G)$. \end{thm}
In our proof of Theorem \ref{cloneupper} we will not only prove that our desired set is a successful zero forcing set, but also precisely describe the forcing pattern which occurs. 
\begin{proof} 
    Let $G$ be any graph and let $H = ILT_l(G)$ for some $l$. Let $U$ be a zero forcing set of $G$ and let $W$ be the subset of $V(H)$ containing $U$ along with every descendant of $U$. We claim that $W$ is a zero forcing set of $H$ and thus $Z(H) \le 2^l Z(G)$. 
    
    Let $u, v \in V(G)$ such that $ u \in U$, $v \notin U$ and a valid first force in $G$ is $u$ forces $v$. We will describe how to force $v$ and all of its descendants in $H$ exactly using $u$ and all of its descendants; this process can then be iterated mirroring each force in $G$ with a series of forces in $H$ until the entirety of $H$ is forced. 

    Assume $x \in V(H)$ is a descendant of $y \in V(G)$. Since $y$ is in level $0$ recall $c(x,y)$, the clone distance between $x$ and $y$, is the number of levels containing an ancestor of $x$. The forcing pattern we describe for $v$ and its descendants is broken into two stages:
    \begin{enumerate}
        \item Force $v$ and all its descendants in levels 1 through $l-1$ with the descendants of $u$ in level $l$. 
        \item Force the descendants of $v$ in level $l$ with the descendants of $u$ in levels $0$ through $l-1$. 
    \end{enumerate}
    
    In the first stage, we force the vertices in increasing order of clone distance from $v$. This results in first forcing $v$ itself, followed by all descendants of $v$, in levels $1$ through $l-1$ only, that first have clone distance $1$, then $2$, up to $l-1$. Let $a \in V(H)$ be the descendant of $u$ in level $l$ with $c(a,u)=l$. This requires that $a$ is obtained by successive cloning in every level beginning with vertex $u$. Note that this implies $N_H(a)$ is exactly $N_G(u)$ along with all of the ancestors of $a$. Note that $a$ and all of its ancestors are descendants of $u$ and thus are in $W$. Further, $N_G(u) \setminus \{v\}$ is also entirely in $U$ and thus $W$, since otherwise $u$ would not have been able to force $v$ in $G$. Therefore, in $H$, $a$ can force $v$. 
    
    Now assume, for the sake of induction, that for some $m \ge 0$ we have successfully forced $v$ and all descendants of $v$ in levels $1$ through $l-1$ with clone distance from $v$ at most $m$. We now describe how to force the descendants of $v$, in levels 1 through $l-1$, with clone distance $m+1$. Let $x$ be such a vertex of clone distance $m+1$ from $v$ and let $v_{i_1}, v_{i_2}, \ldots v_{i_{m+1}}$ be the $m+1$ ancestors of $w$. Here, the notation is such that $v_{i_j}$ is the ancestor of $v$ in level $i_j$, and $i_1=0$ with $v_{i_1} = v_0=v$. As all of the ancestors of $x$ have clone distance at most $m$ from $v$ we may assume they are already forced. 
    
    Let $y$ be the descendant of $u$ in level $l$ with ancestors in every level \emph{except} levels $i_2, i_3, \ldots i_{m+1}$, recall these are the levels containing $x$'s ancestors. Note that this uniquely defines $y$. Furthermore, since each level forms an independent set if $y$ has an ancestor in a given level that ancestor is the only neighbor of $y$ in that level. Thus for every $a \in N(y)\setminus W$ we know $a$ must be a descendant of $v$ with all ancestors in some subset of levels $i_1, i_2, \ldots, i_m$. The maximum possible clone distance from $v$ among the vertices in $N(y)\setminus W$ is $m+1$ and thus $x$ is the unique vertex in $N(y)\setminus W$ whose clone distance from $v$ is exactly $m+1$. Thus $x$ is the only unforced neighbor of $y$ and $y$ can force $x$ as desired.
    
  Now we may assume that the first stage is complete and $v$ and all of its descendants in levels $1$ through $l-1$ are forced. Recall, in the second stage we will now force the descendants of $v$ in level $l$. Again we will do this in increasing order of clone distance from $v$, where the minimum possible clone distance is 1, witnessed by $v'_l$, the clone of $v$ in level l. Here we force $v'_l$, with $c$, the descendant of $u$ in level $l-1$ with clone distance $l-1$ from $u$. Again note that this requires $c$ to have ancestors in every possible level, 0 through $l-2$. Thus, $c$ does not have any neighbors in levels $1$ through $l-1$ other than its ancestors, which are all descendants of $u$. Thus the only descendant of $v$ that $c$ is adjacent to is $v'_l$, as desired. Furthermore, as before all the other neighbors of $c$ are descendants of a vertex in $N_G(u) \setminus \{v\}$ and thus are in $W$, allowing $c$ to force $v'_l$. 
  
  Now, for the sake of induction, assume for some fixed $m \ge 1$ that all the descendants of $v$ in level $l$ with clone distance at most $m$ have been forced. We will now show that all the descendants of $v$ in level $l$ with clone distance $m+1$ from $v$ can now be forced. Let $x$ be such a vertex with ancestors $v_{k_1}, v_{k_2}, \ldots v_{k_{m+1}}$ and let $n$ be the largest integer in $[0, l-1]$ such that $n \notin \{k_1, \ldots k_{m+1}\}$. Pick $y$ to be the descendant of $u$ in level $n$ with ancestors in every level prior to level $n$ except $k_1, \ldots, k_{m+1}$. Now note that vertices in level $l$ which are in $N(y) \setminus W$ are exactly the descendants of $v$ whose ancestors are all contained in, some subset of, levels $k_1, \ldots, k_{m+1}$. Thus, other than $x$, all the vertices in $N(y) \setminus W$ in level $l$ have clone distance at most $m$ from $v$. Therefore, we may assume that $N(y) \setminus \{x\}$ is already forced and $y$ may force $x$ as desired. 

  At this point, $W$ along with $v$ and all of it's descendants are forced. We may then repeat this argument by mirroring the next force which would have taken place in $G$. By continuing this process we will eventually force all of $H$ proving that $W$ is a successful zero forcing set of $H$. 
\end{proof}

Currently our best lower bound for $Z(ILT_l(G))$ comes from Corollary 1 in \cite{throttling} which states that the average degree is a lower bound on the zero forcing number of a graph. For instance, in the case of $ILT_l(K_2)$ this provides a lower bound of $|V(ILT_l(K_2))|^{\log_2(3/2)}$ in contrast with the upper bound from Theorem \ref{cloneupper} of $\frac{|V(ILT_l(K_2))|}{2}$. Note that in the case of $ILT_l(K_2)$ the average degree lower bound is a significant improvement on the trivial minimum degree lower bound as the minimum degree is $\log(|V(ILT_l(K_2))|)$.
In running numerical checks for zero forcing on $ILT_l(K_2)$ for small $l$ we have seen that, at least for small $l$, the upper bound of $\frac{|V(ILT_l(K_2))|}{2}$ is the true zero forcing number of $ILT_l(K_2)$, and we conjecture that this is true for all $l$.

When considering zero forcing on ILAT graphs, however, we see significantly different behavior. 

\begin{thm}\label{anticlonezero}
    For any graph $G$ and $l \ge 2$ we have 
    \begin{equation}\label{ZILAT-ineq}
        \frac{|V(ILAT_l(G))|}{2}-1\le Z(ILAT_l(G)) \le \frac{3|V(ILAT_l(G))|}{4}.
    \end{equation}
\end{thm}
\begin{proof}
    To see the upper bound let $H = ILAT_l(G)$, let $H_{i}$ denote the set of vertices in $H$ in level $i$, and let $H_{\le i}$ denote the set of vertices in $H$ in levels $0$ through $i$. Let $(H_{l-2})'_l = \{v\in V(H): v=u'_l \text{ for some }u \in H_{l-2}\}$. We claim that $Z=V(H)\setminus (H_{l-2} \cup (H_{l-2})'_l)$ is a zero forcing set of $H$. In fact, $Z$ causes all of $V(H)$ to become forced in only two time steps. In the first time step we note that for every $v \in H_{l-2}$ we have that $(v'_{l-1})'_l$ is adjacent to a unique vertex in $Z$, namely $v$. Thus all the vertices in $H_{l-2}$ become forced in the first time step. Once this has occurred we note that for each $v \in H_{l-2}$ we have $v'_{l-1}$ is adjacent to every vertex in $H_{l-2}$ except $v$. Thus each such $v'_{l-1}$ can now force each corresponding $v'_{l}$ simultaneously. 

    For the lower bound we first note that each vertex in $H$, except those in level $l$, have degree $\frac{|V(ILAT_l(G))|}{2}-1$. Thus if $W$ is a zero forcing set of $H$ and $w \in W$ is not in level $l$ and successfully forces a vertex in the first time step then $|W|\ge \frac{|V(ILAT_l(G))|}{2}-1$. Thus we may now assume all of the forces in the first time step are done by vertices in the final level. Each vertex in the final level has degree at least $\frac{|V(ILAT_l(G))|}{4}$, and as level $l$ is itself an independent set this would mean $\left|\bigcup_{i=0}^{l-1} H_i \cap W\right| \ge \frac{|V(ILAT_l(G))|}{4}-1$. Furthermore, if $H_l \subseteq W$ then we have $|W|\ge\frac{|V(ILAT_l(G))|}{2}-1$ and are done. If not, then there is some vertex $v \in H_l \setminus W$ which is the first among the level $l$ vertices to be forced. Let $x \in \bigcup_{i=0}^{l-1} H_i$ be the vertex which forces $v$. Note that $N_{H_l}(x)\setminus\{v\} \subset W$. Thus since each vertex $x \in \bigcup_{i=0}^{l-1} H_i$ has $\deg_{\bigcup_{i=0}^{l-1} H_i}(x)\ge \frac{|V(ILAT_l(G))|}{4}$ we have, in both cases, $|W|\ge \frac{|V(ILAT_l(G))|}{2}-1$ as desired. 
 \end{proof}

\section{Graph Burning}\label{burn-sec}
Graph Burning is a deterministic time process on graphs that models the spread of information, viruses, or social contagions. This model was first introduced by Roshanbin in her thesis, and in a subsequent paper with her supervisors \cite{ogburning, elhamthesis}. In this process, we begin with a graph $G$, and note that all of its vertices are considered \emph{unburned}. At time step $t$ a vertex $v$ is chosen as a new \emph{source} and consider this vertex burned. Then, at time step $t+1$ the fire spreads, and all vertices adjacent to a burned vertex, become burned. That is, we let $B_t$ denote the set of all burned vertices at time $t$ then $B_{t+1} = N[B_t] \cup \{v_{t+1}\}$, as an abuse of notation, where $v_{t+1}$ is the new source chosen at time $t+1$. We denote the minimum number of time steps, $\ell$, required so that $B_\ell = V(G)$, the \emph{burning number} of a graph and denote it $b(G)$.

In this section, we consider the burning number of the iterated models, specifically IIM graphs with mild restrictions. Recall that the ILM and IIM models are generalizations of the ILT and ILAT models in the prior sections. Specifically, the ILM graphs are initiated with some starting graph $G=G_0$, and at each time step $t \geq 1$ built $G_t$ from $G_{t-1}$ by either cloning every vertex in $G_t$ or anticloning every vertex in $G_t$ \cite{ilm}. Furthermore, for IIM graphs at each time step we independently choose, for each previously existing vertex, whether to clone or anticlone at that time step\cite{IIM}. Thus we no longer have a single ILM or IIM graph created from $G$ after $t$ levels, but rather a family of such graphs each resulting from different cloning/anti-cloning decisions. Here we let $\mathcal{ILM}_t(G)$ denote the family of ILM graphs created from $G$ after $t$ time steps and $\mathcal{IIM}_t(G)$ denote the family of IIM graphs created from $G$ after $t$ time steps. Note $ \mathcal{ILM}_t(G) \subset \mathcal{IIM}_t(G)$.

In our proofs, we require a slightly stronger classification for graph burning. Notice, that in this process the fire spreads, and then a new source is chosen at the end of the time step. That is, the fire spreads on the next time step from when the source is chosen. We wish to make a distinction between graphs where the final vertex chosen is redundant or not. That is, had we not selected the additional vertex, would the fire spread to that vertex on this time step. An alternative way of thinking of this, would be that in the last step of the process, the final source node is superfluous. We use the notation $b^*(G) $, when indicating that this bound results in a superfluous final source.

\begin{thm}\label{thrm-burning-edit}
    Let $G$ be some graph and let $H$ be some connected graph created using the IIM process such that at least one vertex in $H$ is an anticlone. That is, $ H \in \mathcal{IIM}_l(G) \setminus \{ILT_l(G)\}$ for some $l$. Then, 
    \[b(G) \leq 4 \]
\end{thm}

To prove Theorem \ref{thrm-burning-edit} we require the following two lemmas. Lemma \ref{lem-IIMburn} demonstrates the power of the single anticlone on the burning process, while Lemma \ref{lem-iltburn} demonstrates the impact of solely cloning. 

\begin{lemma}\label{lem-IIMburn}
    Let $G$ be any graph and let $H$ be some connected graph in $\mathcal{IIM}_l(G)$ which contains an anticlone in the final level $l$. Then $b^*(G) \le 4$.
\end{lemma}

\begin{proof}

    As earlier let $H_i$ denote the vertices in $H$ added in time step $i$. Let $u \in \bigcup _{i=0}^{l-1}H_{i}$ such that $u'_l$ is an anticlone of $u$. In the first time step of our burning processes we burn $u$ followed by burning $u'_l$ at the second time step. When the fire spreads in the third time step all vertices in levels $0$ to $l-1$ will be burned, as every such vertex is in the closed neighborhood of either $u$ or $u'_l$. Since $H$ is connected, every vertex in level $l$ is adjacent to some vertex in $\bigcup _{i=0}^{l-1}H_{i}$, which are all burned by the third burning step. Thus here not only is our choice of new source in the fourth burning step superfluous the choice of source in the third burning step is also immaterial. 
\end{proof}

\begin{lemma}\label{lem-iltburn}
    Let $G$ be a graph with $b(G) = b^*(G)$, and let $H=ILT_l(G)$. Then $b(H) = b(G)$. Furthermore, if every successful burning sequence of $G$ in $b(G)$ steps has a necessary final source then $b(H) = b(G)+1$.
\end{lemma}

\begin{proof}
    Consider in $G$ that a vertex $v$ burns vertex $u$ in time step $t$, then in $H$ we see that $v$ will burn both $u$ and every clone of $u$ in the same time step $t$. We can then easily see the burning process for the ILT graph occurs simultaneously with the burning process for the base graph. Furthermore, note that in ILT graphs if $x'$ is a clone of some vertex $x$ then $N(x') \subset N(x)$. Thus in burning ILT graphs there is never a benefit to burning vertices beyond the base graph. 

    We focus our attention on the final burning step in $G$. Suppose that $b(G) \neq b^*(G)$. This implies that, regardless of the choices of sources used to burn $G$ in $b(G)$ time steps, at the final burning step there is a single vertex $v$ which is non-adjacent to all the vertices in $B_{b(G)-1}$, the vertices burned in the penultimate time step. Now for $H=ILT_l(G)$ we use the same burning pattern as in $G$, but are required to wait one additional time step for the descendants of $v$, the final source in the burning of $G$, to burn. Thus $b(H)=b(G)+1$.

    In the case where the final source vertex in the burning of $G$ is superfluous, we know that every vertex in $G$ must be adjacent to $B_{b(G)-1}$. This remains true in $H$ since if $y$ is a descendant of a given vertex $w\in V(G)$ then $N_G(y)=N_G[w]$.
    Thus in this case $b(G) = b^*(G) = b ^*(H) = b(H)$. 
\end{proof}

We next complete the proof of Theorem~\ref{thrm-burning-edit} using our previous lemmas.

\begin{proof}
    Let $H \in \mathcal{IIM}_l(G)\setminus \{ILT_l(G)\}$ for some graph $G$. Assume $H$ is connected and let level $m$ be the final level containing an anticlone. Let $F$ denote the subgraph of $H$ induced by $\bigcup_{i=0}^m H_i$. Since all remaining levels contain only cloning, $H = ILT_{l-m}(F)$. If $l=m$, the anti-clone is already in the final level, so we are done by Lemma \ref{lem-IIMburn}.

    Since $H$ is connected, and the ILT process preserves connectivity, we know that $F$ must therefore be connected. $F$ has its last level containing an anticlone, thus by Lemma \ref{lem-IIMburn} we know that $b^*(F) \le 4$. Furthermore Lemma \ref{lem-iltburn} now immediately implies $b(H) \le 4$.
\end{proof}

\section{Conclusion and Future Work}\label{con-sec}
In this paper we explored the zero forcing process for ILT and ILAT graphs as well as graph burning for the more general IIM family of graphs. In the investigation of zero forcing we again saw that in the case of the ILAT graphs the anti-cloning procedure soon wipes away the impact of the base graph, while in ILT graphs the cloning procedure retains much of the same structure of the base graph. In particular, the fact that the failed zero forcing number of ILAT graphs with a given number of vertices is one of at most four numbers shows how little the base graph can impact failed zero forcing. However, we see quite different behavior for ILT graphs. Similarly, for zero forcing our best bounds for ILT graphs are in terms of the zero forcing number of the base graph, while for ILAT graphs they are independent. In the case of graph burning we again saw the immense power of anticloning. There we only required \emph{a single anticlone} to bound the burning number by four. 

There are still, of course, many open questions. For instance, is there an infinite family of base graphs such that their ILAT graphs witness the lower bound for failed zero forcing? While such a family is quite likely to exist it would be of interest to see what sort of structure of the base graphs is the cause of the relatively small failed zero forcing number. Furthermore, small numerical checks have supported the notion that for simple base graphs such as paths and cycles the zero forcing number is likely close to or exactly the upper bound given in Theorem \ref{cloneupper}, and thus it would be of interest to prove a matching lower bound in these cases. In the case of zero forcing on ILAT graphs it would be of interest to narrow the interval on which the zero forcing number of such graphs can live. Unlike in the case of ILT graphs, we do not suspect either bounds in Theorem \ref{anticlonezero} are likely the truth for any ILAT graphs. Finally, there remain many other related processes and parameters which model information diffusion on graphs and further exploring how these processes behave on these iterated graph models is of interest.  

\section*{Acknowledgments}
E. Meger acknowledges research support from NSERC (2025-05700) and Queen's University. C. Brice, N-D. Nguyen, and A. Raz acknowledge research support from the Queens Experience in Discrete Mathematics REU funded by the NSF (Award Number DMS 2150251).

\nocite{*}
\bibliographystyle{abbrv}
\bibliography{ZF-ILM-bib}

\end{document}